\documentclass[12pt]{article}

\usepackage[T1]{fontenc}
\usepackage{lmodern}
\usepackage{amsmath,amsfonts,amsthm}
\usepackage[a4paper,margin=37mm]{geometry}
\usepackage{enumitem}
\usepackage[hidelinks]{hyperref}
\hypersetup{
  pdftitle={Continuous Linear Surjections from Cp(X) onto Symmetric Sequence Ideals in c0},
  pdfauthor={Jerzy Kakol and Wieslaw Sliwa},
  pdfsubject={Continuous linear images of Cp-spaces and symmetric sequence ideals},
  pdfkeywords={Cp-space, symmetric sequence ideal, continuous linear surjection, c0, Josefson-Nissenzweig property}
}

\newtheorem{problem}{Problem}[section]
\newtheorem{theorem}{Theorem}[section]
\newtheorem{lemma}[theorem]{Lemma}

\newtheorem{corollary}[theorem]{Corollary}
\newtheorem{definition}[theorem]{Definition}
\newtheorem{remark}[theorem]{Remark}

\newcommand{\N}{\mathbb{N}}
\newcommand{\R}{\mathbb{R}}
\newcommand{\supp}{\operatorname{supp}}
\newcommand{\spancl}{\overline{\operatorname{span}}}

\title{Continuous Linear Surjections from $C_p(X)$ onto Symmetric Sequence Ideals in $c_0$}
\author{Jerzy K\k{a}kol\thanks{Faculty of Mathematics and Computer Science, Adam Mickiewicz University, 61-614 Pozna\'n, Poland. E-mail: \texttt{kakol@amu.edu.pl}.}
\and Wies\l aw \'{S}liwa\thanks{Faculty of Exact and Technical Sciences, University of Rzesz\'ow, 35-310 Rzesz\'ow, Poland. E-mail: \texttt{wsliwa@ur.edu.pl}; \texttt{sliwa@amu.edu.pl}.}}
\date{}

\begin{document}
\maketitle

\begin{abstract}
Rosenthal's   classical theorem says that, for every infinite compact space $X$, the Banach space $C(X)$ admits a quotient isomorphic to either $c_0$ or $\ell_2$.  The corresponding question for $C_p(X)$, the space $C(X)$ endowed with the topology of pointwise convergence, is much subtler and still open; the only compact spaces for which the existence of an infinite-dimensional metrizable quotient is not presently settled in ZFC are Efimov compacta.
We prove that the Banach space case cannot be re-produced with the usual sequence spaces carrying  their pointwise topologies:
Let $X$ be a Tychonoff space and let $E\subseteq c_0$ be a non-trivial symmetric sequence ideal endowed with the subspace topology inherited from $\mathbb{R}^{\mathbb N}$. Then  the existence of a continuous linear surjection $T:C_p(X)\rightarrow E_p$ implies $E=c_0$, where $E_p$ means  $E$  with the topology inherited from $\R^{\N}$. Hence, no proper non-zero symmetric sequence ideal of $c_0$ can be realized as a continuous linear image of a $C_p$-space. Combining this result with the characterization of the Josefson--Nissenzweig property for $C_p(X)$ obtained by Banakh, K\c{a}kol, and \'Sliwa, we derive a complete characterization of all pairs $(X,E)$ for which such a surjection exists. In particular, for every $0<q<\infty$, there is no continuous linear surjection $C_p(X)\rightarrow (\ell_q)_p.$
\end{abstract}

\noindent\textbf{2020 Mathematics Subject Classification.}
Primary 54C35, 46A03; Secondary 46B15, 46B45.

\medskip
\noindent\textbf{Key words and phrases.}
$C_p$-space, continuous linear surjection, symmetric sequence ideal, pointwise topology, Josefson--Nissenzweig property, unconditional basic sequence, Schur property.

\section{Introduction}

 All locally convex spaces in this paper are real and Hausdorff, and all topological spaces are Tychonoff, unless explicitly stated otherwise.  If $X$ is a Tychonoff  space, then $C_p(X)$ denotes the vector space $C(X)$ of all continuous real-valued functions on $X$ endowed with the pointwise topology.

 For a Tychonoff space $X$ and a point $x\in X$ let $\delta_x:C_p(X)\to\mathbb{R}, f\mapsto f(x),$ be the Dirac measure concentrated at $x$. The linear hull $L(X)$ of the set $\{\delta_x:x\in X\}$ in $\mathbb{R}^{C_p(X)}$ can be identified with the dual space of $C_p(X)$, see \cite[Proposition 0.5.7]{Arkh}.

Elements of the space $L(X)$ will be called {\em finitely supported sign-measures} (or simply {\em sign-measures}) on $X$.

Each $\mu\in L(X)$ can be uniquely written as a linear combination of Dirac measures $\mu=\sum_{x\in F}\alpha_x\delta_x$ for some finite set $F\subset X$ and some non-zero real numbers $\alpha_x$. The set $F$ is called the {\em support} of the sign-measure $\mu$ and is denoted by $\supp(\mu)$. The measure $\sum_{x\in F}|\alpha_x|\delta_x$ will be denoted by $|\mu|$ and the real number $\|\mu\|=\sum_{x\in F}|\alpha_x|$ coincides with the {\em norm} of $\mu$ (in the dual Banach space $C(\beta X)^*$).

The theorem of Rosenthal asserts that every infinite-dimensional Banach space $C(K)$ has a quotient isomorphic to $c_0$ or $\ell_2$; see \cite{Rosenthal1969}.  The corresponding problem for pointwise convergence is substantially more delicate.  K\c akol, Saxon  initiated a line of research that was later developed by Banach, K\c akol, \'Sliwa, Marciszewski, Sobota, W\'ojtowicz, Zdomskyy  and others, which asks when spaces $C_p(X)$ admit infinite-dimensional separable, or even metrizable, quotients; see for example  \cite{KakolSliwa2018, K-Saxon, BanakhKakolSliwa2018, BanakhKakolSliwa2019, KakolSliwa2018, MarciszewskiSobotaZdomskyy2025, KakolSobotaZdomskyy2023}.

\begin{problem}[K\c akol--Saxon, K\c akol--\'Sliwa]\label{prob:open}
Let $X$ be an infinite compact space.  Does $C_p(X)$ admit an infinite-dimensional separable quotient?  Does i the space $C_p(X)$  admit an infinite-dimensional metrizable quotient?
\end{problem}
The known positive results reduce a possible negative answer, for compact spaces, to the Efimov case $X$, see \cite{KakolSliwa2018}, where $X$ is an  Efimov compactum  if $X$  contains neither a copy of $\beta\N$ nor a non-trivial convergent sequence.  Indeed, if $X$ contains a non-trivial convergent sequence, then $C_p(X)$ has a complemented subspace isomorphic to $(c_0)_p$, see  \cite[Theorem 1]{BanakhKakolSliwa2019}; and if $X$ contains a copy of $\beta\N$, then $C_p(X)$ has a quotient isomorphic to $(\ell_\infty)_p$, see \cite[Theorem 1]{BanakhKakolSliwa2018}.  Consequently, a compact counterexample would necessarily be of Efimov type.
 The Problem \ref{prob:open}  is also closely connected with the two-disjoint-copies property studied in \cite{KakolKurkaSliwa2026} and \cite{KaniaKakol2026}.

The special role of $(c_0)_p$ is well understood. Banakh, K\k{a}kol, and \'{S}liwa proved that, for every Tychonoff space $X$, the following conditions are equivalent: $C_p(X)$ has the Josefson--Nissenzweig property, $C_p(X)$ contains a complemented copy of $(c_0)_p$, $C_p(X)$ has a quotient isomorphic to $(c_0)_p$, and $C_p(X)$ admits a continuous linear surjection onto $(c_0)_p$,  see \cite[Theorem 1]{BanakhKakolSliwa2018},  see also \cite{KakolSobotaZdomskyy2023},  \cite{MarciszewskiSobotaZdomskyy2025}  for  more recent new results around this property.
By contrast, the usual pointwise sequence spaces occur very abundantly as subspaces. Indeed,  K\k{a}kol, Molt\'o, and \'{S}liwa proved that every infinite $C_p(X)$ contains a subspace isomorphic to $(\ell_q)_p$ for every $0<q\le\infty$; see
\cite[Theorem 3.1]{KakolMoltoSliwa2023}. This makes the corresponding surjection problem particularly natural. For broader results on metrizable quotients of $C_p$-spaces, see \cite{BanakhKakolSliwa2019}.

The purpose of the present paper is to solve that problem for the entire class of symmetric sequence ideals contained in $c_0$. Our main result is the following  theorem.
\begin{theorem} \label{thm:main-intro} Let $X$ be a Tychonoff space and let $E\subseteq c_0$ be a non-zero symmetric sequence ideal. If there exists a continuous linear surjection
 $T:C_p(X)\rightarrow E_p,$ then $E=c_0$.
 \end{theorem}
Thus $(c_0)_p$ is the only non-zero symmetric sequence ideal in $c_0$ which can occur as a continuous linear image of a $C_p$-space.
\begin{corollary}\label{cor:lq}
Let $X$ be a Tychonoff space and let $0<q<\infty$. There is no continuous linear surjection
$C_p(X)\rightarrow(\ell_q)_p.$
\end{corollary}
Here, for $0<q<\infty$, the notation $(\ell_q)_p$ refers to the topology inherited from $\R^{\N}$, not to the norm or quasi-norm topology.

The structural part of the proof of Theorem \ref{thm:main-intro} is isolated in Theorem~\ref{thm:abstract-rigidity}, which may be useful beyond the present setting. Although Theorem \ref{thm:main-intro} already excludes a surjection onto $(c_{00})_p$, the following stronger statement is useful in its own right.  We prove the following mentioned
\begin{theorem}\label{thm:c00} Every continuous linear operator
 $T:C_p(X)\rightarrow(c_{00})_p$
has finite-dimensional range.
\end{theorem}
Together with the characterization from \cite[Theorem 1]{BanakhKakolSliwa2019}, our  Theorem~\ref{thm:main-intro} gives the following complete statement.
\begin{corollary}\label{cor:classification}
Let $X$ be a Tychonoff space and let $E\subseteq c_0$ be a non-zero symmetric sequence ideal. The following conditions are equivalent:
\begin{enumerate}[label=\textup{(\roman*)}]
\item there exists a continuous linear surjection $C_p(X)\to E_p$;
\item $E=c_0$ and $C_p(X)$ has the Josefson--Nissenzweig property;
\item $E=c_0$ and $C_p(X)$ contains a complemented subspace isomorphic to $(c_0)_p$;
\item $E=c_0$ and $C_p(X)$ has a quotient isomorphic to $(c_0)_p$.
\end{enumerate}
\end{corollary}
The space $C_p(X)$ has the \emph{Josefson--Nissenzweig property} if there exists a sequence $(\nu_n)\subseteq C_p(X)^*$ such that
$\|\nu_n\|=1$  for $n\in\N,$
and $\nu_n(f)\rightarrow0$ for $f\in C(X),$ see \cite[Definition 1]{BanakhKakolSliwa2019}.

For $x=(x_n)\in c_0$, define its \emph{decreasing rearrangement}
$x^*=(x_n^*)_{n=1}^{\infty}$ by
\begin{equation}\label{eq:rearrangement}
 x_n^*
 =\inf_{\substack{F\subseteq\N\\ |F|<n}}
   \sup_{k\notin F}|x_k|
 \qquad(n\in\N),
\end{equation}
see, for example, \cite{BennettSharpley}. Thus $x_n^*$ is the $n$th largest modulus in the generalized sense encoded by \eqref{eq:rearrangement}. Then we have
$x_1^*\ge x_2^*\ge\cdots\ge0,\,\,\, x_n^*\rightarrow0.$
The word ``rearrangement'' does not mean that $x^*$ must be obtained by a permutation of all terms of $x$.
\begin{remark}\label{rem:zeros}
Let $x=(x_n)\in c_0.$ Let $s_0=0.$ $\\$
We can choose positive integers $s_1, s_2, s_3, \ldots$ such that $s_n\in (\N\setminus \{s_0, \dots, s_{n-1}\})$ and $|x_{s_n}|=\max_{k\in (\N\setminus \{s_0, \dots, s_{n-1}\})} |x_k|$ for $n\in \N.$ Then  $x^*=(x^*_n)=(|x_{s_n}|)$.
\end{remark}
\begin{remark}
Let $A=\{n\in \N: x_n\neq 0\}.$ If $A=\emptyset$, then $x^*=(0, 0, 0, \ldots )=x$. If $A=\{m_1, \ldots, m_k\}$ is non-empty and finite, and $(|x_{m_{\sigma (1)}}|, \ldots, |x_{m_{\sigma (k)}}|)$ is the non-increasing permutation of the sequence $(|x_{m_1}|, \ldots, |x_{m_k}|)$, then $x^*=(|x_{m_{\sigma (1)}}|, \ldots, |x_{m_{\sigma (k)}}|, 0, 0, 0 \ldots)$. If $A=\{m_k: k\in \N\}$ is infinite, then $x^*=(x^*_n)$ is the non-increasing permutation $(x_{m_{\sigma(k)}})$ of the sequence $(x_{m_k}).$ \end{remark}
We start with the following main
\begin{definition}\label{def:symmetric}
A linear subspace $E\subseteq c_0$ is called a \emph{symmetric sequence ideal} if:
$ x\in E, y\in c_0,$ with $y_n^*\le x_n^*$ for every $n$, then  $y\in E.$
\end{definition}
The definition is algebraic and order-theoretic. We collect a few properties of symmetric sequence ideals.
\begin{lemma}\label{lem:ideal-properties}
Let $E\subseteq c_0$ be a symmetric sequence ideal. Then:
\begin{enumerate}[label=\textup{(\roman*)}]
\item $E$ is solid: if $x\in E$ and $|y_n|\le |x_n|$ for all $n$, then $y\in E$;
\item $E$ is invariant under permutations and changes of signs;
\item every subsequence of a member of $E$ belongs to $E$;
\item if $a\in\ell_\infty$ and $x\in E$, then $(a_nx_n)\in E$;
\item if $E\ne\{0\}$, then $c_{00}\subseteq E$.
\end{enumerate}
\end{lemma}
\begin{proof}
Assertions (i)--(iii) follow directly from the inequalities between decreasing rearrangements. For (iv),
\[
 \bigl((a_nx_n)\bigr)^*_k\le\|a\|_\infty x_k^*
 \qquad(k\in\N),
\]
and $\|a\|_\infty x\in E$.

To prove (v), choose $0\ne x\in E$. Then $x_1^*>0$, and symmetry gives $x_1^*e_1\in E$. Hence $e_1\in E$. Permutation invariance gives $e_n\in E$ for every $n$, and linearity yields $c_{00}\subseteq E$.
\end{proof}

\section{Theorem \ref{thm:functionally-bounded} and its proof}\label{sec:ideals}
Recall that a  subset $A\subseteq X$ is called \emph{functionally bounded} if every $f\in C(X)$ is bounded on $A$.
In order to prove  Theorem   \ref{thm:main-intro} we will need the following crucial
\begin{theorem}\label{thm:functionally-bounded}
Let $(\mu_n)\subseteq C_p(X)^*$ satisfy
$\sup_{n\in\N}|\mu_n(f)|<\infty$ for every $f\in C(X).$
Then the set
$ S=\bigcup_{n=1}^{\infty}\supp\mu_n$ is functionally bounded in $X$.  Moreover, for every $f\in C(X)$ there exists $g\in C_b(X)$ such that
$\mu_n(g)=\mu_n(f)$ for $n\in\N.$
\end{theorem}
\begin{proof}
The set $S$ is at most countable. We first choose a function $h\in C_b(X)$ which is one-to-one on $S$. If $S$ is finite, this is clear. If $S$ is infinite, for distinct $x,y\in S$ put
\[
 H_{x,y}=\{u\in C_b(X):u(x)=u(y)\}.
\]
Each $H_{x,y}$ is a closed proper linear subspace of the Banach space $C_b(X)$ and hence has empty interior. Since the set $S \times S$ is countable, the Baire category theorem gives
\[
 C_b(X)\neq \bigcup_{x\ne y}H_{x,y}.
\]
Any function $h \in (C_b(X)\setminus\bigcup_{x\ne y}H_{x,y})$ is one-to-one on $S$.

Fix $f\in C(X)$ and define
\[
 \Phi:X\rightarrow\R^2,
 \qquad
 \Phi(x)=(f(x),h(x)).
\]
Let $C_k(\R^2)$ denote the Fr\'echet space of continuous functions on $\R^2$ with the compact-open topology. For every $n$, define
\[
 \lambda_n:C_k(\R^2)\rightarrow\R,
 \qquad
 \lambda_n(\varphi)=\mu_n(\varphi\circ\Phi).
\]
Because $\mu_n$ is finitely supported, $\lambda_n$ is continuous. For every $\varphi\in C_k(\R^2)$, the scalar sequence
\[
 \bigl(\lambda_n(\varphi)\bigr)_{n=1}^{\infty}
 =\bigl(\mu_n(\varphi\circ\Phi)\bigr)_{n=1}^{\infty}
\]
is bounded. Hence $(\lambda_n)$ is pointwise bounded. The Banach--Steinhaus theorem yields an equicontinuity estimate: there exist a compact set $K_0 \subset \R^2$ and $C>0$ such that
\begin{equation}\label{eq:equicontinuity}
 |\lambda_n(\varphi)|
 \le C\sup_{z\in K_0}|\varphi(z)|
 \qquad \mbox{for all}\; n\in\N,\ \varphi\in C_k(\R^2).
\end{equation}

We claim that
\begin{equation}\label{eq:PhiS}
 \Phi(S)\subseteq K_0.
\end{equation}
Suppose otherwise. Choose $n$ and $x\in\supp\mu_n$ such that $p=\Phi(x)\notin K_0$. Write
\[
 \mu_n=\sum_{y\in F_n}a_y\delta_y, 
\]
with every $a_y\ne0$. The injectivity of $h$ on $S$ implies that the points $\Phi(y)$, $y\in F_n$, are pairwise distinct. Thus
\[
 K=K_0\cup\{\Phi(y):y\in (F_n\setminus\{x\})\}
\]
is a compact subset of $\R^2$ which does not contain $p$. Define
\[
 \varphi(z)=
 \frac{\operatorname{dist}(z,K)}
 {\operatorname{dist}(z,K)+\|z-p\|}.
\]
Then $\varphi$ is continuous, $\varphi(p)=1$, and $\varphi$ vanishes on $K$. The right-hand side of \eqref{eq:equicontinuity} is therefore zero, whereas
\[
 \lambda_n(\varphi)
 =\mu_n(\varphi\circ\Phi)
 =a_x\ne0,
\]
a contradiction. This proves \eqref{eq:PhiS}.

The first coordinate projection of the compact set $K_0$ is bounded. Since $\Phi(S)\subseteq K_0$,
\[
 \sup_{x\in S}|f(x)|<\infty.
\]
As $f\in C(X)$ was arbitrary, $S$ is functionally bounded.

Let $f\in C(X)$. Choose $M>0$ with $|f(x)|\le M$ for $x\in S$ and let
\[
 \rho_M(t)=\max\{-M,\min\{t,M\}\}.
\]
Then $g=\rho_M\circ f$ belongs to $C_b(X)$ and agrees with $f$ on $S$. Every $\mu_n$ is supported by $S$, so $\mu_n(g)=\mu_n(f)$ for every $n\in \N$.
\end{proof}

\section{An abstract Banach-space theorem}\label{sec:abstract}
In this section  the sequence ideal is now supplemented by any Banach norm consistent with its natural inclusion in $c_0$.  The following Theorem \ref{thm:abstract-rigidity}  will also be used in the proof of  the main
Theorem \ref{thm:main-intro}.
\begin{theorem}\label{thm:abstract-rigidity}
Let $E\subseteq c_0$ be a non-zero symmetric sequence ideal. Suppose that $E$ is equipped with a Banach norm $\|\cdot\|_E$ such that the inclusion map
$$J:(E,\|\cdot\|_E)\rightarrow(c_0,\|\cdot\|_\infty)$$
is continuous. Let $Z$ be a Banach space and let
$$Q:Z\rightarrow E$$
be a bounded surjective operator. For $n\in\N$, let $\delta_n: E \to \R$  be defined $\delta_n(x)=x_n$. Assume that there exists a closed subspace $M$ of $Z^*$ with the Schur property such that
$Q^*\delta_n\in M$ for $n\in\N.$
Then $E=c_0$ as a set.
\end{theorem}
The proof will be divided into several lemmas. Since $E\ne\{0\}$, Lemma~\ref{lem:ideal-properties} gives $c_{00}\subseteq E$.

\begin{lemma}\label{lem:diagonal}
For every $a=(a_n)\in\ell_\infty$, the diagonal multiplier
\[
 D_a:E\rightarrow E,
 \qquad
 D_a(x)=(a_nx_n),
\]
is bounded. Moreover,
\[
 K_D:=\sup_{\|a\|_\infty\le1}\|D_a\|<\infty.
\]
\end{lemma}

\begin{proof}
By Lemma~\ref{lem:ideal-properties}, $D_a(E)\subseteq E$. Fix $a\in\ell_\infty$. If $x_j\to x$ in $E$ and $D_ax_j\to y$ in $E$, then continuity of $J$ implies both convergences in $c_0$. Since $D_a$ is bounded on $c_0$, we have $D_ax_j\to D_ax$ in $c_0$, and therefore $y=D_ax$. Thus $D_a$ has closed graph, so it is bounded.

Fix $x\in E$ and consider
\[
 M_x:\ell_\infty\rightarrow E,
 \qquad
 M_x(a)=D_ax.
\]
If $a^{(j)}\to a$ in $\ell_\infty$ and $M_x(a^{(j)})\to y$ in $E$, then
\[
 \|D_{a^{(j)}}x-D_ax\|_\infty
 \le\|a^{(j)}-a\|_\infty\|x\|_\infty\rightarrow0.
\]
The continuity of $J$ again implies $y=D_ax$. Hence $M_x$ has closed graph, so it is bounded. Consequently,
\[
 \sup_{\|a\|_\infty\le1}\|D_ax\|_E<\infty
 \qquad(x\in E).
\]
The uniform boundedness principle applied to the family $\{D_a:\|a\|_\infty\le1\}$ gives the assertion.
\end{proof}

\begin{lemma}\label{lem:permutations}
Every permutation of the coordinates induces a bounded automorphism of $E$. If $n_1<n_2<\cdots$, then
$R:E\rightarrow E,\,\, R(x)=(x_{n_k})_{k=1}^{\infty}$ is bounded.
\end{lemma}
\begin{proof}
Algebraic invariance follows from symmetry. The closedness of the graph follows from the continuity of $E\rightarrow c_0$ and the corresponding continuity of the coordinate operation on $c_0$. The closed graph theorem, see \cite[Theorem 11.1.7]{Jarchow},  gives the boundedness. By the same argument to the inverse permutation we obtain  a bounded automorphism. For a subsequence, $(x_{n_k})^*\le x^*$, so the map is well defined; finally the same closed graph argument applies.
\end{proof}

For every $n$, the coordinate functional $\delta_n$ belongs to $E^*$ and $\|\delta_n\|\leq \|J\|$, because
$\delta_n(x)|\le\|x\|_\infty\le\|J\|\,\|x\|_E.$
Put
\[
 F=\spancl\{\delta_n:n\in\N\}\subseteq E^*.
\]
We need also the following
\begin{lemma}\label{lem:unconditional}
The sequence $(\delta_n)$ is a seminormalized unconditional Schauder basis of $F$.
\end{lemma}

\begin{proof}
Clearly, every $\delta_n$ is non-zero and $\sup_n \|\delta_n\|\leq \|J\|<\infty$.  Suppose that $\inf_n\|\delta_n\|=0$. We can choose positive integers $q_1, p_1, q_2, p_2, \ldots$ with $q_1<p_1<q_2<p_2<\ldots$ such that for every $k\in \N$ we have
\[
 \|\delta_{q_k}\|>k\|\delta_{p_k}\|.
\] Let $\sigma$ be a permutation of $\N$ such that $\sigma (q_k)=p_k$ and $\sigma (p_k)=q_k$ for every $k\in \N$.
By Lemma~\ref{lem:permutations}, the corresponding permutation operator $P_\sigma$ is bounded. Since
\[
 P_\sigma^*\delta_{p_k}=\delta_{q_k},
\]
for every $k\in \N$, we obtain
\[
 \|\delta_{q_k}\|\le\|P^*_\sigma\|\,\|\delta_{p_k}\|,
\]
a contradiction for sufficiently large $k$. Thus $\inf_n\|\delta_n\|>0$.

For $a\in\ell_\infty$, the adjoint $D_a^*$ leaves $F$ invariant and
\[
 D_a^*\delta_n=a_n\delta_n.
\]
By Lemma~\ref{lem:diagonal}, these operators are uniformly bounded when $\|a\|_\infty\le1$.

Let $n,m \in \N$ with $n\leq m$, let $\varepsilon_1,
\ldots, \varepsilon_n \in \{-1, 1\}, a_1, \ldots, a_m \in \R$ and $\varepsilon=(\varepsilon_1, \ldots, \varepsilon_n, 0, 0, \ldots)$. Then $D^*_{\varepsilon}(\sum_{k=1}^m a_k \delta_k)= \sum_{k=1}^n \varepsilon_k a_k\delta_k$, so
\[\|\sum_{k=1}^n \varepsilon_k a_k\delta_k\|\leq \|D^*_{\varepsilon}\| \|\sum_{k=1}^m a_k \delta_k\|\leq K_D \|\sum_{k=1}^m a_k \delta_k\|.\]

It follows that $(\delta_n)$ is an unconditional Schauder basis of $F$,  by \cite[Proposition 4.36]{Fabian}. \end{proof}


We use also the following known standard dichotomy for unconditional basic sequences \cite[Lemma 4.1]{Laustsen}.


\begin{lemma}\label{lem:dichotomy}
Let $(u_n)$ be a seminormalized unconditional basic sequence in a Banach space. Then either $(u_n)$ is weakly null, or some subsequence of $(u_n)$ is equivalent to the canonical basis of $\ell_1$.
\end{lemma}
We are ready to prove Theorem \ref{thm:abstract-rigidity}.
\begin{proof}[Proof of Theorem~\ref{thm:abstract-rigidity}]
We apply Lemma~\ref{lem:dichotomy} to $(\delta_n)$ in $F$.

Assume first that $\delta_n\to0$ weakly in $F$. Every element of $E^{**}$ restricts to a continuous functional on $F$, so the sequence is also weakly null when regarded as a sequence in $E^*$. Therefore
\[
 Q^*\delta_n\rightarrow0
 \quad\text{weakly in }Z^*.
\]
All these functionals belong to the closed subspace $M$. By Hahn--Banach, weak convergence in $Z^*$ restricts to weak convergence in $M$, and the Schur property gives
\[
 \|Q^*\delta_n\|\rightarrow0.
\]
Since $Q$ is a bounded surjection between Banach spaces, the open mapping theorem implies that $Q^*$ is bounded below: there exists $c>0$ such that
\[
 \|Q^*\psi\|\ge c\|\psi\|
 \qquad(\psi\in E^*).
\]
This contradicts the inequality $\inf_n\|\delta_n\|>0$.

Thus the second alternative holds: There is a strictly increasing sequence $(n_k)$ and $c>0$ such that
\begin{equation}\label{eq:l1-subsequence}
 c\sum_{k=1}^r|a_k|
 \le
 \left\|\sum_{k=1}^r a_k\delta_{n_k}\right\|
\end{equation}
for every finite scalar sequence $(a_1, \ldots, a_r)$. Let
\[
 R:E\to E,
 \qquad
 R(x)=(x_{n_k})_{k=1}^{\infty}.
\]
By Lemma~\ref{lem:permutations}, $R$ is bounded, and
\[
 R^*\delta_k=\delta_{n_k}.
\]
Hence \eqref{eq:l1-subsequence} gives
\[
 \frac{c}{\|R^*\|}\sum_{k=1}^r|a_k|
 \le
 \left\|\sum_{k=1}^r a_k\delta_k\right\|.
\]
Together with the inequality $\sup_n\|\delta_n\|<\infty$, this shows that the sequence $(\delta_n)$ is equivalent to the canonical basis of $\ell_1$. Thus there are constants $\alpha,M>0$ such that
\begin{equation}\label{eq:l1-full}
 \alpha\sum_{n=1}^r|a_n|
 \le
 \left\|\sum_{n=1}^r a_n\delta_n\right\|
 \le
 M\sum_{n=1}^r|a_n|.
\end{equation}

Let
\[
 E_0=\overline{c_{00}}^{\|\cdot\|_E}.
\]
The coordinate projections
\[
 P_N: E \to E, P_Nx=(x_1,\ldots,x_N,0,0,\ldots)
\]
are diagonal multipliers, and Lemma~\ref{lem:diagonal} gives
\[
 K_P:=\sup_N\|P_N\|<\infty.
\]
It follows by density that $P_Nx\to x$ in $E$ for every $x\in E_0$. Thus $(e_n)$ is a Schauder basis of $E_0$, with coefficient functionals $\delta_n|_{E_0}$.

Let $\varphi\in E_0^*$ and extend it by Hahn--Banach to $\widetilde\varphi\in E^*$ with the same norm. Then
\[
 P_N^*\widetilde\varphi
 =\sum_{n=1}^N\varphi(e_n)\delta_n.
\]
Using the lower estimate in \eqref{eq:l1-full},
\[
 \alpha\sum_{n=1}^N|\varphi(e_n)|
 \le\|P_N^*\widetilde\varphi\|
 \le K_P\|\varphi\|.
\]
For $x=\sum_{n=1}^N x_ne_n\in c_{00}$, this yields
\[
 |\varphi(x)|
 \le\|x\|_\infty\sum_{n=1}^N|\varphi(e_n)|
 \le\frac{K_P}{\alpha}\|x\|_\infty\|\varphi\|.
\]
Taking the supremum over the unit ball of $E_0^*$ gives
\begin{equation}\label{eq:E-sup-upper}
 \|x\|_E\le\frac{K_P}{\alpha}\|x\|_\infty
 \qquad(x\in c_{00}).
\end{equation}
The continuity of $J$ gives the reverse estimate
\begin{equation}\label{eq:E-sup-lower}
 \|x\|_\infty\le\|J\|\,\|x\|_E
 \qquad(x\in E).
\end{equation}
Hence the two norms are equivalent on $c_{00}$.

Let $x=(x_n)\in c_0$ and let $x^{(N)}=(x_1,\ldots,x_N,0,\ldots)$. The sequence $(x^{(N)})$ is Cauchy in the supremum norm, and therefore, by \eqref{eq:E-sup-upper}, Cauchy in $E$. It converges in the Banach space $E_0$ to some $y$. By \eqref{eq:E-sup-lower}, it also converges to $y$ in $c_0$, while it converges to $x$ in the supremum norm. Thus $y=x$. Hence $c_0\subseteq E_0\subseteq E\subseteq c_0$, and consequently $E=c_0$.
\end{proof}

\section{Proof of main Theorem   \ref{thm:main-intro}} \label{sec:main-proof}

\begin{proof}[Proof of Theorem \ref{thm:main-intro}] For $n\in\N$, let $\pi_n:E_p\to\R$ be the $n$-th coordinate functional and put
\[
 \mu_n=\pi_n\circ T\in C_p(X)^*.
\]
Then
\[
 T(f)=\bigl(\mu_n(f)\bigr)_{n=1}^{\infty}
 \qquad \mbox{for} f\in C(X).
\]
Since $E\subseteq c_0\subset l_{\infty}$, Theorem~\ref{thm:functionally-bounded} shows that
\[
 S=\bigcup_{n=1}^{\infty}\supp\mu_n
\]
is functionally bounded and every $f\in C(X)$ can be replaced by $g\in C_b(X)$ with the same values under all $\mu_n$. Consequently,
\[
 T(C_b(X))=T(C(X))=E.
\]
Define
\[
 T_b:C_b(X)\rightarrow c_0,
 \qquad
 T_b(g)=\bigl(\mu_n(g)\bigr)_{n=1}^{\infty}.
\]
Its range is $E$. We claim that $T_b:C_b(X)\to c_0$ is bounded. If $g_j\to g$ in $C_b(X)$ and $T_bg_j\to y$ in $c_0$, then, for every $n$,
\[
 y_n=\lim_j\mu_n(g_j)=\mu_n(g),
\]
so $y=T_bg$. Thus the graph of $T_b$ is closed, so $T_b$ is bounded.

Equip $E$ with the quotient norm
\[
 \|x\|_E
 =\inf\{\|g\|_\infty:T_bg=x\}.
\]
Then $E$ is isometrically isomorphic to the Banach quotient
\[
 C_b(X)/\ker T_b,
\]
and the natural inclusion $E\hookrightarrow c_0$ is bounded because
\[
 \|x\|_\infty\le\|T_b:C_b(X)\to c_0\|\,\|x\|_E.
\]
With this norm, $T_b:C_b(X)\to E$ is a bounded surjection.

Let
\[
 M_S=\spancl\{\delta_x:x\in S\}\subseteq C_b(X)^*.
\]
For every finite $A\subseteq S$ and scalars $(a_x)_{x\in A}$ we have
\[
 \left\|\sum_{x\in A}a_x\delta_x\right\|
 =\sum_{x\in A}|a_x|.
\]
Therefore $M_S$ is isometrically isomorphic to $\ell_1(S)$ and has the Schur property. Moreover,
$T_b^*\delta_n=\mu_n\in M_S$
for
 $n\in\N.$
All hypotheses of Theorem~\ref{thm:abstract-rigidity} are satisfied with
\[
 Z=C_b(X),\qquad Q=T_b\;\;\; \mbox{and}\qquad M=M_S.
\]
Hence $E=c_0$.
\end{proof}
\begin{corollary}\label{cor:no-proper}
Let $X$ be a Tychonoff space and let $E$ be a non-zero proper symmetric sequence ideal in $c_0$. Then there is no continuous linear surjection
$C_p(X)\rightarrow E_p.$
\end{corollary}
Corollary~\ref{cor:no-proper} excludes arbitrary continuous linear images, not merely quotient maps. Note that  a continuous linear surjection between locally convex spaces need not be open.

\section{Proofs of Theorem \ref{thm:c00} and corollaries}\label{sec:boundary}
\begin{proof}[Proof of Theorem \ref{thm:c00}]
Let
\[
 T(f)=\bigl(\mu_n(f)\bigr)_{n=1}^{\infty},
 \qquad (\mu_n)\subset C_p(X)^*.
\]
Since $T(f)\in c_{00}\subset l_{\infty}$ for every $f$, Theorem~\ref{thm:functionally-bounded} implies that
\[
 S=\bigcup_n\supp\mu_n
\]
is functionally bounded. Hence
\[
 T(C(X))=T(C_b(X)).
\]
For $N\in\N$, put
\[
 F_N=\bigcap_{n>N}\ker(\mu_n|_{C_b(X)}).
\]
Each $F_N$ is a closed linear subspace of the Banach space $C_b(X)$. Since $T(g)\in c_{00}$ for every $g\in C_b(X)$, there exists $N$ depending on $g$ such that $\mu_n(g)=0$ for all $n>N$. Hence
\[
 C_b(X)=\bigcup_{N=1}^{\infty}F_N.
\]
By the Baire category theorem, some $F_N$ has non-empty interior. A linear subspace with non-empty interior is the whole space, so
\[
 \mu_n|_{C_b(X)}=0
 \qquad \mbox{for}\;\;\; n>N.
\]
Since each $\mu_n$ is a finite linear combination of point evaluations and $C_b(X)$ separates finite sets, this implies $\mu_n=0$ on $C(X)$ for $n>N$. Thus
\[
 T(C_p(X))\subseteq\operatorname{span}\{e_1,\ldots,e_N\},
\]
so the range of $T$ is finite-dimensional.
\end{proof}


\begin{proof}[Proof of Corollary \ref{cor:lq}]
For every $0<q<\infty$, the space $\ell_q$ is a non-zero symmetric sequence ideal and is a proper subspace of $c_0$. Corollary~\ref{cor:no-proper} applies.
\end{proof}

\begin{proof}[Proof of Corollary \ref{cor:classification}]
The implication (i)$\Rightarrow$(ii) follows from Theorem~\ref{thm:main-intro}, the equivalence between a surjection onto $(c_0)_p$ and the Josefson--Nissenzweig property established in \cite[Theorem~1]{BanakhKakolSliwa2018}. The equivalence of (ii)--(iv), and their implication to (i), are precisely the remaining parts of that theorem.
\end{proof}

\begin{corollary}\label{cor:some-X}
Let $E\subseteq c_0$ be a non-zero symmetric sequence ideal. Then $E=c_0$ if and only if for some Tychonoff space $X$, there exists a continuous linear surjection $C_p(X)\to E_p$.
\end{corollary}




\end{document}